\documentclass[centertags]{conm-p-l}

\usepackage{amssymb}	
\usepackage{eucal}	

\newcommand{\defeq}{\mathrel{:=}}

\renewcommand\ge\geqslant
\renewcommand\le\leqslant

\def\Hl{\mathop{{Hl}}\nolimits}

\makeatletter

\newenvironment{sizedisplaymath}[1]{%
  \skip@=\baselineskip
  #1%
  \baselineskip=\skip@
  \displaymath
}{\enddisplaymath \ignorespacesafterend}
\newenvironment{sizemultline}[1]{%
  \skip@=\baselineskip
  #1%
  \baselineskip=\skip@
  \multline
}{\endmultline \ignorespacesafterend}  

\theoremstyle{plain}
\newtheorem{mytheorem}{Theorem}[section]
\newtheorem{myproposition}[mytheorem]{Proposition}

\newtheorem{mycorollary}{Corollary}[mytheorem]
\newtheorem{mycorollarycorollary}{Corollary}[mycorollary]

\theoremstyle{definition}
\newtheorem{mydefinition}[mytheorem]{Definition}
\newtheorem*{myexample*}{Example}

\theoremstyle{remark}
\newtheorem{myremarkaftertheorem}{Remark}[mytheorem]
\newtheorem*{myremark*}{Remark}

\begin{document}

\title{P-symbols, Heun identities, and ${}_3F_2$ identities}
\author[R. S. Maier]{Robert S. Maier}
\address{Depts.\ of Mathematics and Physics\\ University of Arizona\\ Tucson,
AZ 85721\\ USA}
\email{rsm@math.arizona.edu}
\urladdr{http://math.arizona.edu/\~{}rsm}
\date{}

\begin{abstract}
The usefulness of Riemann P-symbols in deriving identities involving the
parametrized special function~$\Hl$ is explored.  $\Hl$~is the analytic
local solution of the Heun equation, the canonical second-order
differential equation on the Riemann sphere with four regular singular
points.  The identities discussed include ones coming from M\"obius
automorphisms and F-homotopies, and also quadratic and biquadratic
transformations.  The case when $\Hl$~is identical to a generalized
hypergeometric function of ${}_3F_2$~type is examined, and Pfaff and Euler
transformations of ${}_3F_2(a_1,a_2,e+1;b_1,e;x)$ are derived.  They extend
several ${}_3F_2$ identities of Bailey and Slater.
\end{abstract}

\copyrightinfo{2007}
  {American Mathematical Society}

\maketitle

\section{Introduction}
\label{sec:intro}

The Gauss hypergeometric equation (GHE) and the Heun equation
(HE)~\cite{Ronveaux95} are canonical second-order Fuchsian differential
equations on the Riemann sphere~$\mathbb{P}^1_x$ with three and four
singular points, respectively.  (A~differential equation is Fuchsian if
each of its singular points is regular.)  By convention, they are written
as
\begin{subequations}
\begin{align}
&\left\{D^2 + \left[\frac{\gamma}{x} + \frac{\delta}{x-1}\right] D +
\left[\frac{\alpha\beta}{x(x-1)}\right]\right\}F = 0,\label{eq:GHE}\\ 
&\left\{D^2 +
\left[ \frac\gamma x + \frac\delta{x-1} + \frac\epsilon{x-a} \right]D +
\left[\frac{\alpha\beta\, x - q}{x(x-1)(x-a)}\right]\right\}F = 0,\label{eq:HE}
\end{align}
\end{subequations}
respectively.  Here $\alpha,\beta;\gamma,\delta,$ resp.\
$a(\neq0,1),q;\alpha,\beta;\gamma,\delta,\epsilon$ are complex-valued
parameters; and each equation is invariant
under~$\alpha\leftrightarrow\beta$.  The singular points are
$x=0,1,\infty,$ resp.\ $x=0,1,a,\infty,$ and their characteristic exponents
are
$0,\nobreak1-\gamma;\allowbreak0,\nobreak1-\delta;\allowbreak\alpha,\nobreak\beta,$
resp.\
$0,\nobreak1-\gamma;\allowbreak0,\nobreak1-\delta;\allowbreak0,\nobreak1-\epsilon;\allowbreak\alpha,\nobreak\beta$.
These parameters are constrained by Fuchs's relation on exponents.  The sum
of the exponents of the GHE, resp.\ the~HE, must be $1,$ resp.~$2$.  So
$\delta=\alpha+\beta-\gamma+1,$ resp.\
$\epsilon=\alpha+\beta-\gamma-\delta+1$.  If this Fuchsian condition does
not hold, the singular point $x=\infty$ will be irregular.  The
parameter~$q\in\mathbb{C}$ of the HE is `accessory': it does not affect the
exponents.

If $\gamma$~is a nonpositive integer, the local (Frobenius) solution of the
GHE or HE corresponding to the zero exponent at~$x=0$ will generically be
logarithmic; we do not consider this case.  If $\gamma$~is {\em not\/} a
nonpositive integer, the solution will be analytic at~$x=0$.  When
normalized to unity at~$x=0,$ it is denoted
${}_2F_1(x)\defeq {}_2F_1(\alpha,\beta;\gamma;x),$ resp.\
$\Hl(x)\defeq \Hl(a,q;\alpha,\beta;\gamma,\delta;x),$ and is called the Gauss
hypergeometric function, resp.\ the local Heun function.  (The shorter term
`Heun function' has a more specialized meaning~\cite{Ronveaux95}.)  The
redundant parameter $\delta,$ resp.~$\epsilon,$ is suppressed.

The parametrized special functions ${}_2F_1,\Hl$ have series expansions
of the form $\sum_{n=0}^\infty c(n)x^n$.  The coefficient sequences
$c\colon\mathbb{N}\to\mathbb{C}$ satisfy respective recurrences
\begin{subequations}
\begin{align}
  &\Bigl\{(n+\gamma)(n+1)E - (n+\alpha)(n+\beta)\Bigr\}\,c=0,\label{eq:GHErec}\\
  &\Bigl\{(n+\gamma+1)(n+2)a\,E^2 -\left[(n+1)(n+\gamma+\delta)a +
  (n+1)(n+\gamma+\epsilon) + q\right]E\label{eq:HErec}\\
  &\hskip3.05in{}+ (n+\alpha)(n+\beta)\Bigr\}\,c=0,\notag
\end{align}
\end{subequations}
where $E$~denotes the shift operator, i.e., $(E^kc)(n)\defeq c(n+k)$.  The
recurrence (\ref{eq:GHErec}), resp.~(\ref{eq:HErec}), is initialized by
$c(0)=1,$ resp.\ $c(0)=1,\allowbreak c(-1)=0$.  If $c(n)$ is taken to equal
zero when~$n<0,$ then (\ref{eq:GHErec}),(\ref{eq:HErec}) will hold for
all~$n\in\mathbb{Z}$.  The power series converges on $\left|x\right|<1,$
resp.\ $\left|x\right|<\min(1,\left|a\right|)$.

The special functions ${}_2F_1,\Hl$ play a central role not only in the
theory of second-order Fuchsian differential equations, but also in the
theory of linear recurrences with quadratic coefficients.  From
(\ref{eq:GHErec}),(\ref{eq:HErec}), one easily deduces the following.

\begin{myproposition}
\ 
  \begin{enumerate}
    \item Suppose $c\colon\mathbb{N}\to\mathbb{C}$ {\rm(}with $c(n)\defeq 0$ for
      $n<0$ by convention{\rm)}\/ satisfies a $2$-term recurrence
      $\left[P_1(n)E+P_0(n)\right]c=0,$ where $\deg P_1=\deg P_0=2,$ and
      one root of~$P_1$ is~$-1$ and the other is not an integer greater
      than~$-1$.  Then the ordinary generating function
      {\rm(}o.g.f.\/{\rm)\null} of~$c$ equals
      $C\,{}_2F_1(\alpha,\beta;\gamma;Ax)$ for some
      $A(\neq0),C\in\mathbb{C}$ and some choice of GHE
      parameters~$\alpha,\beta;\gamma$.
    \item Suppose $c\colon\mathbb{N}\to\mathbb{C}$ {\rm(}with $c(n)\defeq 0$ for
      $n<0$ by convention{\rm)}\/ satisfies a $3$-term recurrence
      $\left[P_2(n)E^2+P_1(n)E+P_0(n)\right]c=0,$ where $\deg P_2=\deg
      P_0=2,$ $\deg P_1\le2,$ and one root of~$P_2$ is~$-2$ and the other
      is not an integer greater than~$-2$.  Suppose, moreover, that the
      characteristic polynomial of this recurrence has distinct roots.
      {\rm(}This polynomial is formed from the coefficients of~$n^2$
      in~$P_2,P_1,P_0$: it equals $p_{2,2}n^2 + p_{1,2}n +
      p_{0,2}$.{\rm)}\/ Then the o.g.f.\ of~$c$ equals
      $C\Hl(a,q;\alpha,\beta;\gamma,\delta;Ax)$ for some
      $A(\neq0),C\in\mathbb{C}$ and some choice of HE
      parameters~$a,q;\alpha,\beta;\gamma,\delta$.
  \end{enumerate}  
\end{myproposition}

\noindent
In part~(2), the taking of the polynomial roots to be distinct avoids the
case when the differential equation satisfied by the o.g.f.\ is confluent,
i.e., non-Fuchsian.

We explain here how the classical formalism of Riemann P-symbols
(see~\S\,\ref{sec:psymbols}) facilitates the derivation of {\em Heun
identities\/}, such as alternative expressions for the local
function~$\Hl,$ written in~terms of itself.  In~\S\,\ref{sec:autos}, we begin
by determining the automorphism groups of the GHE and~HE\null.  We showed
in a previous treatment~\cite{Maier10} that the two automorphism groups, of
orders $24$ and~$192$ respectively, are the Coxeter groups $\mathcal{D}_3$
($\cong S_4$) and~$\mathcal{D}_4$.  (The set of $24$ GHE automorphisms was
first worked~out by Kummer, but its group structure long remained obscure;
an isomorphism to the octahedral group~$S_4$ was pointed~out by
Dwork~\cite{Dwork84} and others~\cite{Becken2000,Lievens2005}.)
$\mathcal{D}_3$ and~$\mathcal{D}_4$ have subgroups $\mathcal{D}_2$
and~$\mathcal{D}_3,$ which are the transformation groups of ${}_2F_1$
and~$\Hl$.  The former, of order~$4,$ is generated by classical
transformations of Pfaff and Euler, but the latter, of order~$24,$ is a
novel object.  In~\cite{Maier10}, we computed the $24$ resulting
expressions for~$\Hl(x)$ with the aid of a computer algebra system.  We
show here how with the aid of P-symbols, they may be computed by~hand.

In~\S\,\ref{sec:quadratics}, we compute quadratic and quartic transformations
of~$\Hl,$ which express $\Hl(x)$ in~terms of $\Hl(R(x)),$ where $R$~is
rational of degree $2$ or~$4$.  It is known that the derivation of the
quadratic transformations of~${}_2F_1$ is facilitated by the P-symbol
formalism~\cite[\S3.9]{Andrews99}.  Erd\'elyi~\cite[\S15.3]{Erdelyi53}
pointed~out that it can also be used to derive quadratic and quartic HE
transformations, but did not compute explicit formulae\ involving~$\Hl$.
Up~to automorphism, we find one quadratic transformation of $\Hl$ and one
quartic one that is biquadratic, i.e., is the composition of two
quadratics.  The $\Hl$ situation apparently contrasts with that
of~${}_2F_1$: not~all the quartic transformations of~${}_2F_1,$ which were
worked~out by Goursat~\cite{Goursat1881}, are biquadratic.

Using P-symbols, we investigate in~\S\,\ref{sec:Hlto3F2} the interesting case
when $\Hl$ coincides with a generalized hypergeometric function,
in~particular with a~${}_3F_2$ of the form ${}_3F_2(a_1,a_2,e+1;b_1,e;x)$.
If the parameters $a,q;\alpha,\beta;\gamma,\delta$ are suitably chosen,
this phenomenon will occur, as noted by Letessier, Valent, and
Wimp~\cite{Letessier94}.  Since the coefficients of the series expansion
of~${}_3F_2$ satisfy a first-order recurrence, the equality is attributable
to a factorization of the second-order difference operator
in~(\ref{eq:HErec}).  Such factorizations can be computed
algorithmically~\cite{Cluzeau2006,Petkovsek92}.  We give an explicit one,
and in the same spirit, factor the third-order differential operator in the
equation satisfied by the~${}_3F_2$.  The operator products in the two
factorizations can be viewed as {\em desingularizations\/} in the
sense of Abramov, Barkatou, and van Hoeij~\cite{Abramov2006}.

In~\S\,\ref{sec:3F2}, we derive transformations analogous to Pfaff's and
Euler's for the function ${}_3F_2(a_1,a_2,e+1;b_1,e;x),$ involving
nonlinear parametric constraints.  Bailey and Slater (see, e.g.,
Slater~\cite[\S\,2.4.2]{Slater66}) found an Pfaff-like identity satisfied
by~${}_3F_2,$ with three free parameters, but our identity, which extends
theirs, has four.  As a corollary, we extend other results of Bailey; e.g.,
his result that any very well poised~${}_3F_2$ can be written in~terms of a
well poised~${}_2F_1,$ and vice versa.

\section{Riemann P-symbols}
\label{sec:psymbols}

Any P-symbol tabulates the characteristic exponents of a linear
(homogeneous) Fuchsian differential equation (FDE) $\mathcal{L}u=0$
on~$\mathbb{P}^1$.  It partially characterizes the space of local
solutions.  The columns of the P-symbol list the exponents associated to
each singular point of~$\mathcal{L}$.  For instance, the GHE and HE have
P-symbols
\begin{equation}
\label{eq:Psymbols}
\left\{\!
\begin{array}{ccc|c}
0 & 1 & \infty & x \\
\hline
0 & 0 & \alpha &  \\
1-\gamma & 1-\delta & \beta & \\
\end{array}
\!\right\},
\qquad
\left\{\!
\begin{array}{cccc|c}
0 & 1 & a & \infty & x \\
\hline
0 & 0 & 0 & \alpha & \\
1-\gamma & 1-\delta & 1-\epsilon & \beta & \\
\end{array}
\!\right\},
\end{equation}
where the column and exponent orderings are not significant.  Ordinary
(non-singular) points may optionally be included, each with its own column.
Any finite ordinary point of a second-order FDE has exponents~$0,1,$ and
the point $x=\infty,$ if ordinary, has exponents~$0,-1$.  Having `ordinary'
exponents is not sufficient for ordinariness, since zero-exponent Frobenius
solutions may be logarithmic.

Any second-order FDE on~$\mathbb{P}^1$ with $n\ge3$ singular points is
determined up~to equivalence by its $2n$~exponents and $n-3$ accessory
parameters.  Though the latter parameters are not shown, P-symbols
facilitate the computation of transformed FDEs, as the following indicates.
Automorphisms of~$\mathbb{P}^1$ are M\"obius transformations (also called
homographies or linear fractional transformations) of the form
$M(x)=(Ax+B)/(Cx+D),$ where $A,B,C,D\in\mathbb{C}$ and $AD\neq BC$.  Under
such transformations, singular points are accompanied by their exponents.
That is, if $M\colon\mathbb{P}^1_x\to\mathbb{P}^1_{x'},$ applying the
change of variable ${x'}=M(x)$ to an FDE $\mathcal{L}u=0$
on~$\mathbb{P}^1_{x'}$ lifts it to $(M^*\mathcal{L})v=0$
on~$\mathbb{P}^1_x,$ where if $\mathcal{L}$ has singular points
$a',b',c',\dots,$ with respective exponents
$\alpha_1,\alpha_2;\beta_1,\beta_2;\gamma_1,\gamma_2,\dots,$ the operator
$M^*\mathcal{L}$ will have P-symbol
\begin{equation}
\label{eq:homography}
\left\{\!
\begin{array}{cccc|c}
M^{-1}a' & M^{-1}b' & M^{-1}c' & \ldots & x \\
\hline
\alpha_1 & \beta_1 & \gamma_1 & \ldots & \\
\alpha_2 & \beta_2 & \gamma_2 & \ldots & \\
\end{array}
\!\right\}
=
\left\{\!
\begin{array}{cccc|c}
a' & b' & c' & \ldots & M(x) \\ 
\hline
\alpha_1 & \beta_1 & \gamma_1 & \ldots & \\ 
\alpha_2 & \beta_2 & \gamma_2 & \ldots & \\
\end{array}
\!\right\}.
\end{equation}
Fuchs's relation states that the sum of the $2n$ exponents of any
second-order FDE with $n$~singular points equals~$n-2$.  M\"obius lifting
trivially preserves this.

Since there is a (unique) M\"obius transformation that takes any three
distinct points to any other three, one can normalize any FDE with $n\ge3$
singular points by moving three of them to~$0,1,\infty$.  Multiplying the
dependent variable by~$(x-x_0)^{-\zeta}$ will decrement the exponents
at~$x=x_0$ by~$\zeta,$ and those at~$x=\infty$ by~$-\zeta$; i.e.,
\begin{equation}
\label{eq:Fhomotopy}
(x-x_0)^{-\zeta}
\left\{\!
\begin{array}{ccc|c}
x_0 & \ldots & \infty & x \\ 
\hline
\theta_1 & \ldots & \eta_1 & \\ 
\theta_2 & \ldots & \eta_2 & \\
\end{array}
\!\right\}
=
\left\{\!
\begin{array}{ccc|c}
x_0 & \ldots & \infty & x \\ 
\hline
\theta_1-\zeta & \ldots & \eta_1+\zeta & \\ 
\theta_2-\zeta & \ldots & \eta_2+\zeta & \\
\end{array}
\!\right\}
\end{equation}
Using~(\ref{eq:Fhomotopy}), one can further normalize the FDE by shifting
one exponent at each finite singular point to zero.  For $n=3,4,$ the
resulting FDE is the GHE, resp.~HE\null.

Any nonconstant rational function ${x'}=R(x)$ specifies a $(\deg
R)$-sheeted covering $R\colon\mathbb{P}^1_x \to \mathbb{P}^1_{x'}$.
Applying the change of variables ${x'}=R(x)$ to an FDE $\mathcal{L}u=0$
on~$\mathbb{P}^1_{x'}$ lifts it to an FDE $(R^*\mathcal{L})v=0$
on~$\mathbb{P}^1_x$.

\begin{myproposition}
\label{prop:lifting}
  {\rm(i)}~The characteristic exponents of\/~$R^*\mathcal{L}$ at any
  point\/ $x_0\in R^{-1}{x'_0}$ equal those of\/~$\mathcal{L}$ at\/~${x'_0},$
  multiplied by the multiplicity with which\/ $x_0$~is mapped to\/~${x'_0}$.
  {\rm(ii)}~If the exponents of\/~$\mathcal{L}$ at\/~${x'_0}$ are\/ $0,1/k$
  for some integer\/~$k>1,$ and\/ $x_0$~is mapped to\/~${x'_0}$ with
  multiplicity\/~$k,$ then\/ $x_0$~will be an ordinary point
  of\/~$R^*\mathcal{L}$.
\end{myproposition}
\begin{proof}
  Part~(i) follows from the definition of a characteristic exponent.  To
  prove~(ii), it suffices to show that the zero-exponent Frobenius solution
  of $(R^*\mathcal{L})v=0$ at~$x=x_0$ cannot be logarithmic.  By the
  definition of a lifting, this solution $v=v(x)$ (normalization being
  arbitrary) can be written locally as $u(R(x)),$ where $u=u(x)$ is a
  zero-exponent Frobenius solution of $\mathcal{L}u=0$.  But such a
  solution can be logarithmic only if the other exponent is a positive
  integer; and here it is~$1/k$.
\end{proof}

By examination, lifting $\mathcal{L}$ to~$R^*\mathcal{L}$ preserves Fuchs's
relation.  But under a non-M\"obius rational lifting, singular points may
disappear or appear.  Part~(ii) of the proposition indicates how an inverse
image of a singular point may fail to be singular.  Also, if ${x'_0}$~is an
ordinary point of~$\mathcal{L},$ with $x_0\in R^{-1}{x'_0}$ mapped by~$R$
to~${x'_0}$ with multiplicity $k>1,$ then $x_0$~will be a singular point
of~$R^*\mathcal{L},$ with exponents~$0,k$.

With the aid of Prop.~\ref{prop:lifting}, one can conjecture and verify
rational transformation laws for the special functions ${}_2F_1$ and~$\Hl$.
However, in any such identity the accessory parameter~$q$ of the HE will
need to be considered separately.  It should be mentioned that there are
types of Heun identity, such as differential transformations, on which
P-symbols throw less light.  One example is the following.

\begin{mytheorem}[\cite{Umetsu2000},\S3]
  For every integer $N\ge0,$ one has the differential Heun identity
  \begin{displaymath}
    D^N\Hl(a,q;\,1-N,\beta;\,\gamma,\delta;\,x) \propto
    \Hl(a,q';\,1+N,\beta+N;\,\gamma+N,\delta+N;\,x),
  \end{displaymath}
  where $q'\defeq q + N(N-1)(a+1) + N\left[(a+1)\gamma + a\delta + \epsilon\right]$.
\end{mytheorem}

\noindent
This result can be strengthened to an equality between HE P-symbols, namely
\begin{multline}
\label{eq:extended}
D^N 
\left\{\!
\begin{array}{cccc|c}
0 & 1 & a & \infty & x \\
\hline
0 & 0 & 0 & 1-N & \\
1-\gamma & 1-\delta & 1-\epsilon & \beta & \\
\end{array}
\!\right\}\\
=
\left\{\!
\begin{array}{cccc|c}
0 & 1 & a & \infty & x \\
\hline
0 & 0 & 0 & 1+N & \\
1-(\gamma+N) & 1-(\delta+N) & 1-(\epsilon+N) & \beta+N & \\
\end{array}
\!\right\},
\end{multline}
which says that any local solution of the left-hand HE, differentiated
$N$~times, is one of the right-hand~HE\null.  The P-symbol
identity~(\ref{eq:extended}) is plausible, since the exponents on left and
right formally correspond.  For instance, at~$x=0$ the left-hand Frobenius
solutions are of the form $x^0$ and $x^{1-\gamma}$ times analytic
functions, each nonzero at~$x=0$; so (generically only!\null), their $N$'th
derivatives are of the form $x^0$ and $x^{1-(\gamma+N)}$ times such
analytic functions.  But a satisfactory proof of~(\ref{eq:extended}) along
these lines would require a study of the series coefficients of the
Frobenius solutions, since if any non-leading coefficient were zero, it
would yield non-generic behavior.  And of~course, $q'$~needs to be computed
separately.  In the following sections the focus is on non-differential
identities, which are easier to conjecture and prove.

\section{Automorphisms of the GHE and HE}
\label{sec:autos}

P-symbols of normalized FDEs with $n\ge3$ singular points can be written in
a symmetrical way.  Without loss of generality, the singular points may be
taken to include $x=0,\infty$; and typically, $x=1$ as~well.  Each finite
singular point may be taken to have a zero exponent.  The P-symbol of such
an FDE will be of the form
\begin{equation}
  \left\{\!
\begin{array}{ccccc|c}
a_1 & a_2 & \ldots & a_{n-1} & \infty & x \\
\hline
0 & 0 & \ldots & 0 & \alpha & \\
\theta_1 & \theta_2 & \ldots & \theta_{n-1} & \beta\defeq \alpha+\theta_\infty & \\
\end{array}
\!\right\},
\end{equation}
where $a_1=0$ (and usually, $a_2=1$).  The canonical FDE solution~$F$
at~$x=a_1=0$ will be written
$F(a_2,\dots,a_{n-1};q;\theta_1,\dots,\theta_{n-1},\theta_\infty;x),$ where
$q$~is a suitably defined vector of $n-3$~accessory parameters.
Necessarily $\alpha=(n-2-\theta_1-\cdots-\theta_{n-1}-\theta_\infty)/2$.

The automorphism group of such normalized FDEs, or P-symbols, contains a
subgroup of automorphisms that negate exponents, or more accurately
exponent {\em differences\/}.  These are called
F-homotopies~\cite{Ronveaux95}.  The subgroup is generated by
\begin{displaymath}
(x-a_i)^{-\theta_i}
\left\{\!
\begin{array}{cccc|c}
\ldots & a_i & \ldots & \infty & x \\
\hline
\ldots & 0 & \ldots & \alpha & \\
\ldots & \theta_i & \ldots & \alpha+\theta_\infty & \\
\end{array}
\!\right\}
=
\left\{\!
\begin{array}{cccc|c}
\ldots & a_i & \ldots & \infty & x \\
\hline
\ldots & 0 & \ldots & \alpha+\theta_i & \\
\ldots & -\theta_i & \ldots & \alpha+\theta_\infty+\theta_i & \\
\end{array}
\!\right\},
\end{displaymath}
$i=1,\dots,n-1,$ and by the trivial $\alpha\leftrightarrow\beta$
interchange at~$x=\infty,$ i.e.,
\begin{displaymath}
\left\{\!
\begin{array}{cccc|c}
a_1 & \ldots & a_{n-1} & \infty & x \\
\hline
0 & \ldots & 0 & \alpha & \\
\theta_1 & \ldots & \theta_{n-1} & \alpha+\theta_\infty & \\
\end{array}
\!\right\}
=
\left\{\!
\begin{array}{cccc|c}
a_1 & \ldots & a_{n-1} & \infty & x \\
\hline
0 & \ldots & 0 & \alpha+\theta_\infty & \\
\theta_1 & \ldots & \theta_{n-1} & \alpha & \\
\end{array}
\!\right\}.
\end{displaymath}
The group of F-homotopies is isomorphic to $(\mathbb{Z}_2)^n$; or
$(\mathbb{Z}_2)^{n-1},$ if $\alpha\leftrightarrow\beta$ is excluded.

An F-homotopy at any of the $n-1$ singular points other than $x=a_1=0$
yields a transformation of the canonical local solution $F=F(x)$ at
$x=a_1=0$.  Taking the usual normalization $F(x=a_1=0)=1$ into account, one
deduces
\begin{align}
  &F(a_2,\dots,a_{n-1};\,q;\,\theta_1,\dots,\theta_i,\dots,\theta_{n-1},\theta_\infty;\,x)\label{eq:nontrivial}\\
  &\qquad=(1-x/a_i)^{\theta_i}
  F(a_2,\dots,a_{n-1};\,q';\,\theta_1,\dots,-\theta_i,\dots,\theta_{n-1},\theta_\infty;\,x),\notag
\end{align}
$i=2,\dots,n-1,$ where for each $i,$ the transformed vector~$q'$ of
accessory parameters must be computed separately, by manipulating the
FDE\null.  One also has
\begin{align}
  &\hskip-0.44in F(a_2,\dots,a_{n-1};\,q;\,\theta_1,\dots,\theta_{n-1},\theta_\infty;\,x)\label{eq:trivial}\\
  &\qquad\qquad=F(a_2,\dots,a_{n-1};\,q;\,\theta_1,\dots,\theta_{n-1},-\theta_\infty;\,x).\notag
\end{align}
The group of F-homotopic transformations of~$F$ is isomorphic
to~$(\mathbb{Z}_2)^{n-1}$; or to~$(\mathbb{Z}_2)^{n-2},$ if the trivial
transformation~(\ref{eq:trivial}), i.e., $\alpha\leftrightarrow\beta,$ is
excluded.

The GHE case $n=3,$ in which there are no accessory parameters, illustrates
this.  The identity~(\ref{eq:nontrivial}), when $n=3$ and~$i=2,$ says that
\begin{equation}
\label{eq:euler}
{}_2F_1(\alpha,\beta;\,\gamma;\,x) = (1-x)^{\gamma-\alpha-\beta}\,{}_2F_1(\gamma-\alpha,\gamma-\beta;\,\gamma;\,x),
\end{equation}
and (\ref{eq:trivial})~simply says that
${}_2F_1(\alpha,\beta;\gamma;x)={}_2F_1(\beta,\alpha;\gamma;x)$.
Equation~(\ref{eq:euler}) is Euler's transformation
of~${}_2F_1$~\cite{Andrews99}, which is an involution that commutes
with~$\alpha\leftrightarrow\beta$.

The automorphism group of normalized FDEs $\mathcal{L}u=0$ with
$n$~singular points, or strictly of their P-symbols, also contains a
subgroup of M\"obius automorphisms, based on M\"obius transformations.  If
$\mathcal{L}u=0$ is to be lifted to $(M^*\mathcal{L})v=0$ along
$M\colon\mathbb{P}_x^1\to\mathbb{P}_{x'}^1,$ the inverse images of the
singular points ${x'}=0,1,\infty$ of~$L$ must be chosen to be singular
points of $M^*\mathcal{L}$.  So there are $n(n-1)(n-2)$ possibilities, and
there is no~additional freedom in the choice of~$M$.  But since the order
of the remaining $n-3$ columns in the lifted P-symbol is not significant,
one may permute them arbitrarily.  So the M\"obius subgroup is isomorphic
to the symmetric group~$S_n$.

The cases $n=3,4$ illustrate this.  In the $n=3$ GHE case, $x'=M(x)$
may be any of $3\cdot2\cdot1=6$ homographies.  The correspondence to
permutations of $0,1,\infty$ is
\begin{displaymath}
\label{eq:GHEmob}
  \begin{array}{cccccc}
    x, & \frac{x}{x-1}, & 1-x, & \frac{x-1}{x}, & \frac{1}{x}, & \frac{1}{1-x}\\
    (0)(1)(\infty), & (0)(1\infty), & (01)(\infty), & (0\infty1), &
    (0\infty)(1), & (01\infty).\\
  \end{array}
\end{displaymath}
In the $n=4$ HE case, $x'=M(x)$ may be any of
$4\cdot3\cdot2=24$~homographies, given in~\cite{Maier10}.  The six that
stabilize $x=0,$ which are of primary interest here, are
\begin{displaymath}
\label{eq:HEmob}
  \begin{array}{cccccc}
    x, & \frac{x}{x-1}, & \frac{(1-a)x}{x-a}, & \frac{(a-1)x}{a(x-1)}, &
    \frac{x}{a}, & \frac{x}{x-a} \\
    (0)(1)(a)(\infty), & (0)(1\infty)(a), & (0)(1)(a\infty), & (0)(1\infty
    a), & (0)(1a)(\infty), & (0)(1a\infty) \\
    a, & \frac{a}{a-1}, & 1-a, & \frac{a-1}{a}, & \frac{1}{a}, & \frac{1}{1-a}.\\  \end{array}
\end{displaymath}
Each~$M$ corresponds to a permutation of the HE singular points
$0,1,a,\infty,$ i.e., columns, as shown.  But the permutation is to be
interpreted in a special way.  For instance, $(0)(1a\infty)$ means
$0\mapsto0,$ $1\mapsto a',$ $a\mapsto\infty,$ $\infty\mapsto1,$ for
some~$a'$.  In~general, $a\in\mathbb{P}^1_x$ and $a'\in\mathbb{P}^1_{x'}$
differ.  The value of~$a',$ in~terms of~$a,$ is given in the third row.

A M\"obius automorphism based on a lifting along ${x'}=M(x)$ must in~general
include an exponent shift, of the type shown in~(\ref{eq:Fhomotopy}), in
order to keep one exponent zero at each finite singular point.  The
formulas for the M\"obius tranformations of the local solution $F=F(x)$
at~$x=0$ will include such factors.  Any of the above $n(n-1)(n-2)$
M\"obius transformations that fixes the point $x=a_1=0$ will yield a
transformation of~$F$.  If the map ${x'}=M(x)$ is not of the form ${x'}=Ax,$ it
must be of the form $M(x)=Ax/[x-M^{-1}(\infty)],$ where $M^{-1}(\infty)$~is
one of~$a_2,\dots,a_{n-1}$.  Taking $F(x=a_1=0)=1$ into account, one
obtains a transformation of the type
\begin{multline}
\label{eq:masterM}
  F(a_2,\dots,a_{n-1};\, q;\,
  \theta_1,\dots,\theta_{n-1},\theta_\infty;\,x)\\
  =
  \left[1-x/M^{-1}(\infty)\right]^{-\alpha}
  F(a_2',a_3',\dots,a'_{n-1};\, q';\,
  \theta_1',\dots,\theta'_{n-1},\theta'_\infty;\,M(x)),
\end{multline}
in which the transformed vector~$q'$ of accessory parameters must be
computed separately.  Any such formula is really $(n-3)!$ formally distinct
transformations of~$F,$ since the $n-3$~parameter pairs
$(a_3,\theta_3),\dots,(a_{n-1},\theta_{n-1})$ on the left-hand side may be
permuted arbitrarily.  (Equivalently, the columns of the left-hand
P-symbol, other than the $x=0,1,\infty$ columns, may be permuted.)  The
group of M\"obius transformations of~$F$ is isomorphic to~$S_{n-1}$.

The $n=3$ GHE case is illustrative.  Of the $3\cdot2\cdot1=6$ homographies
above, the only two that fix $x=0$ are the identity $x\mapsto x$ and
$x\mapsto\frac{x}{x-1}$; they make~up a group isomorphic to~$S_2$.  By
examination, the latter yields the identity
\begin{equation}
\label{eq:pfaff}
{}_2F_1(\alpha,\beta;\,\gamma;\,x) = (1-x)^{-\alpha}\,{}_2F_1(\alpha,\gamma-\beta;\,\gamma;\,\tfrac{x}{x-1}),
\end{equation}
valid near~$x=0$.  This is Pfaff's transformation
of~${}_2F_1$~\cite{Andrews99}, which is an involution.

Any M\"obius automorphism permutes the columns of the P-symbol, and any
F-homotopic one acts on columns individually.  So the $S_n$ subgroup of the
former normalizes the $(\mathbb{Z}_2)^n$ subgroup of the latter, and the
full group of automorphisms is isomorphic to a~semidirect product
$(\mathbb{Z}_2)^n\rtimes S_n$.  It is easy to see that this is the wreath
product $\mathbb{Z}_2\wr S_n,$ the group of signed permutations of an
$n$-set.  Any element of $\mathbb{Z}_2\wr S_n\cong(\mathbb{Z}_2)^n\rtimes
S_n$ is an ordered pair $(h,\sigma),$ comprising a M\"obius and an
F-homotopic automorphism, which act from right to left.  The first permutes
the singular points, i.e., columns of the P-symbol, according to $\sigma\in
S_n,$ and the second negates zero or more of the exponent differences
$\theta_1,\dots,\theta_n$.  Any element $(h,\sigma)$ can accordingly be
written in a sign-annotated version of disjoint cycle notation.  For
instance, $[0_+][1_+a_-\infty_-]$ signifies the signed permutation
$0\mapsto0,$ $1\mapsto a',$ $a\mapsto\infty,$ $\infty\mapsto1,$ followed by
negations at $x=\infty$ and~$x=1$.  `Even-signed' permutations, such as
this one, include an even number of minus signs; they play a special role.

\begin{mytheorem}
\label{thm:main}
  The automorphism group of normalized second-order FDEs on\/
  $\mathbb{P}^1$ with $n\ge3$ singular points is the Coxeter group\/
  $\mathcal{B}_n\defeq \mathbb{Z}_2\wr S_n,$ the group of signed permutations of
  an $n$-set. If the transformation $\alpha\leftrightarrow\beta$ is not
  explicitly included as an element, the automorphism group will be the
  Coxeter group\/ $\mathcal{D}_n\defeq \left[\mathbb{Z}_2\wr S_n\right]_{\rm
  even},$ the group of even-signed permutations of an $n$-set.  The
  transformation group of the canonical solution $F=F(x)$ is\/
  $\mathcal{B}_{n-1},$ or\/ $\mathcal{D}_{n-1}$ if
  $\alpha\leftrightarrow\beta$ is not explicitly included.
\end{mytheorem}

\begin{myremark*}
  The orders of these four groups are $\left|\mathcal{B}_n\right|=2^nn!,$
  $\left|\mathcal{D}_n\right|=2^{n-1}n!,$
  $\left|\mathcal{B}_{n-1}\right|=2^{n-1}(n-1)!,$ and
  $\left|\mathcal{D}_{n-1}\right|=2^{n-2}(n-2)!$.  It is readily verified
  that $\mathcal{B}_3,\mathcal{D}_3$ are isomorphic to $\mathbb{Z}_2\times
  S_4$ and~$S_4$; and that $\mathcal{B}_2,\mathcal{D}_2$ are isomorphic to
  the dihedral group of order~$8$ and the Klein four-group
  $(\mathbb{Z}_2)^2$.  The groups $\mathcal{B}_n,\mathcal{D}_n$ are usually
  not defined as permutation groups, but as reflection groups consistent
  with certain Coxeter graphs~\cite{Grove85}.  Due~to this, the notation
  used here is slightly nonstandard: our~$\mathcal{D}_3$ is usually
  denoted~$\mathcal{A}_3,$ and our~$\mathcal{D}_2$ is not usually called a
  Coxeter group.
\end{myremark*}

\begin{proof}
  This theorem follows from the preceding discussion, the transformation
  group of~$F$ comprising all automorphisms that fix the singular point
  $x=0,$ and perform no~F-homotopy there.  If $\alpha\leftrightarrow\beta,$
  i.e., $\theta_\infty\mapsto-\theta_\infty,$ is not explicitly included,
  then one may modify each automorphism with an odd number of F-homotopies
  (negations) by following it by $\theta_\infty\mapsto-\theta_\infty,$ to
  coerce even-signedness.  This alteration is innocuous, since it leaves
  the FDE invariant.
\end{proof}

\begin{myexample*}[$n=3$]
The correspondence between the elements of~$\mathcal{D}_2,$ the group of
even-signed permutations of the nonzero singular points $x=1,\infty$ of the
GHE, and transformations of~${}_2F_1$ is
\begin{subequations}
\label{eq:firsthalf}
\begin{eqnarray}
\label{eq:firsthalffirst}
{}[1_+][\infty_+] & \quad\sim\quad &{}_2F_1(\alpha,\beta;\,\gamma;\,x)\vphantom{\tfrac{x}{x-1}},\\
{}[1_-][\infty_-] & \quad\sim\quad &(1-x)^{\gamma-\alpha-\beta}\,{}_2F_1(\gamma-\alpha,\gamma-\beta;\,\gamma;\,x)\vphantom{\tfrac{x}{x-1}},\label{eq:Euler2}\\
{}[1_+\infty_+] & \quad\sim\quad &(1-x)^{-\alpha}\,{}_2F_1(\alpha,\gamma-\beta;\,\gamma;\,\tfrac{x}{x-1}),\label{eq:Pfaff2}\\
{}[1_-\infty_-] & \quad\sim\quad &(1-x)^{-\beta}\,{}_2F_1(\gamma-\alpha,\beta;\,\gamma;\,\tfrac{x}{x-1}).\label{eq:Pfaff3}
\end{eqnarray}
\end{subequations}
Here (\ref{eq:Euler2}),(\ref{eq:Pfaff2}) are Euler's and Pfaff's
transformations, and (\ref{eq:Pfaff3})~is a twisted Pfaff's transformation
in which the first and second arguments, and also $\alpha$~and~$\beta,$ are
interchanged.  Substituting any of
(\ref{eq:Euler2}),(\ref{eq:Pfaff2}),(\ref{eq:Pfaff3}) for $F(a,b;c;x)$ is
an involutory operation.  Any two of these three operations commute, and
their product is the third, confirming that $\mathcal{D}_2$~is isomorphic
to the Klein four-group.  If the transformation group of~${}_2F_1$ is
extended by explicitly including the involution
$\alpha\leftrightarrow\beta,$ the group becomes not~$\mathcal{D}_2,$
but~$\mathcal{B}_2$.  The nontrivial coset of~$\mathcal{D}_2$
in~$\mathcal{B}_2$ comprises
\begin{subequations}
\label{eq:secondhalf}
\begin{eqnarray}
\label{eq:secondhalffirst}
{}[1_+][\infty_-] & \quad\sim\quad &{}_2F_1(\beta,\alpha;\,\gamma;\,x)\vphantom{\tfrac{x}{x-1}},\\
{}[1_-][\infty_+] & \quad\sim\quad &(1-x)^{\gamma-\alpha-\beta}\,{}_2F_1(\gamma-\beta,\gamma-\alpha;\,\gamma;\,x)\vphantom{\tfrac{x}{x-1}},\\
{}[1_+\infty_-] & \quad\sim\quad &(1-x)^{-\beta}\,{}_2F_1(\beta,\gamma-\alpha;\,\gamma;\,\tfrac{x}{x-1}),\\
{}[1_-\infty_+] & \quad\sim\quad &(1-x)^{-\alpha}\,{}_2F_1(\gamma-\beta,\alpha;\,\gamma;\,\tfrac{x}{x-1}).
\label{eq:secondhalflast}
\end{eqnarray}
\end{subequations}
Examination confirms that as expected, the set of transformations
of~${}_2F_1(x)$ in (\ref{eq:firsthalf}),(\ref{eq:secondhalf}), under
composition, is isomorphic to the order-$8$ dihedral group.

The automorphism group of the GHE is isomorphic to~$\mathcal{D}_3\cong S_4$
(of order~$24$), if $\alpha\leftrightarrow\beta$ is not included.  The
group~$\mathcal{D}_2$ yielding
(\ref{eq:firsthalffirst})\textendash(\ref{eq:Pfaff3}) is a proper subgroup;
a~GHE automorphism may be more general than a transformation of~${}_2F_1$.
In~all, $\mathcal{D}_3$ yields $24$~solutions of the GHE, which are the
well-known solutions of Kummer. They are partitioned into $6$~equivalence
classes of size~$4,$ one of which is
(\ref{eq:firsthalffirst})\textendash(\ref{eq:Pfaff3}).
\end{myexample*}

\begin{myexample*}[$n=4$]
For any of the six homographies $x'=M(x)$ given above, which in~effect
permute the nonzero singular points $x=1,a,\infty$ of the HE, one can
rewrite (\ref{eq:masterM}) in HE notation, obtaining a M\"{o}bius
transformation of~$\Hl$.  The painful part is computing the transformed
accessory parameter.  The usual parametrization~(\ref{eq:HE}) of the HE is
not adapted to this, and a closely related one is better.

First, write $q=\alpha\beta Q,$ where $Q$~is an alternative accessory
parameter.  ($Q,$ not~$q,$ was Heun's accessory parameter; the present
convention was introduced by Erd\'elyi.)  Second, perform a change of
variable, $\bar x=a/x,$ i.e., $(0\infty)(1a),$ accompanied by an F-homotopy
that restores the transformed HE to HE form.  In P-symbol terms,
\begin{multline}
x^{\alpha}\left\{\!
\begin{array}{cccc|c}
0 & 1 & a & \infty & x \\
\hline
0 & 0 & 0 & \alpha & \\
1-\gamma & 1-\delta & 1-\epsilon & \beta & \\
\end{array}
\!\right\}\\
=
\left\{\!
\begin{array}{cccc|c}
0 & 1 & a & \infty & a/x \\
\hline
0 & 0 & 0 & \alpha & \\
\beta-\alpha & 1-\epsilon & 1-\delta & \alpha-\gamma+1 & \\
\end{array}
\!\right\}.
\end{multline}
That is, if $F=v(x)$ is a solution of the HE~(\ref{eq:HE}), then $x^\alpha
v(x),$ regarded as a function of $\bar x=a/x,$ i.e., as a multiple of
${\bar x}^{-\alpha} v(a/\bar x),$ will satisfy a~HE of the form
\begin{equation}
\label{eq:HeunHE}
\left\{D_{\bar x}^2 +
\left[ \frac{\bar\gamma}{\bar x} + \frac{\bar\delta}{\bar x-1} +
  \frac{\bar\epsilon}{\bar x-a} \right]D_{\bar x} +
\left[\frac{\bar\alpha\bar\beta(\bar x - \bar Q)}{\bar x(\bar x-1)(\bar x-a)}\right]\right\}F = 0,
\end{equation}
with parameters $\bar\alpha=\alpha,$ $\bar\beta=\alpha-\gamma+1,$
$\bar\gamma=\alpha-\beta+1,$ $\bar\delta=\epsilon,$ $\bar\epsilon=\delta$.
By direct computation, $\bar Q=\left[\beta
Q+(\epsilon-\beta)a+(\delta-\beta)\right]/(\alpha-\gamma+1)$.

In computing the M\"{o}bius transformations of~$\Hl,$ $\bar x=a/x$ is a
more natural independent variable than~$x,$ and $\bar Q$~a more natural
parameter than~$Q$ or~$q$.  This is because in~terms of~$\bar x,$ each of
the six homographies is {\em affine linear\/}, since it fixes the point
$\bar x=\infty,$ i.e., $x=0$.  From each homography $x'= M(x),$ the affine
map $\bar x'= \bar M(\bar x)$ follows from $\bar x'=\bar M(\bar x)= a'/x' =
a'/M(a/\bar x)$.  The correspondence is
\begin{displaymath}
  \begin{array}{rcccccc}
    M(x): & x, & \frac{x}{x-1}, & \frac{(1-a)x}{x-a}, & \frac{(a-1)x}{a(x-1)}, & \frac{x}{a}, & \frac{x}{x-a}\\
    \bar M(\bar x): & \bar x, & \frac{a-\bar x}{a-1}, & 1-\bar x, &
    \frac{a-\bar x}{a}, & \frac{\bar x}{a}, & \frac{1-\bar x}{1-a}.
\end{array}
\end{displaymath}
And from~(\ref{eq:HeunHE}), under any of these affine changes of variable
$\bar x'=\bar M(\bar x),$ it is clearly the case that $\bar Q'=\bar M(\bar
Q)$.  So, $\bar Q'$ in~terms of~$\bar Q$ is quite easy to calculate.

The homography $M(x)=\frac{x}{x-1},$ which comes from the permutation
$(1\infty)(a)$ of the nonzero singular points $x=1,a,\infty,$ is typical.
For this homography, Eq.~(\ref{eq:masterM}) says that in a neighborhood
of~$x=0,$ $\Hl(a,q;\alpha,\beta;\gamma,\delta;x)$ can equivalently be
written as
$(1-x)^{-\alpha}\Hl(a',q';\alpha',\beta';\gamma',\delta';\frac{x}{x-1}),$
in which $a'=\frac{a}{a-1},$ $\alpha'=\alpha,$ $\beta'=\alpha-\delta+1,$
$\gamma'=\gamma,$ and $\delta'=\alpha-\beta+1,$ as is easily verified.
What remains to be determined is the transformed parameter~$q'$.  But $\bar
Q'=(a-\bar Q')/(a-1),$ since the affine linear map associated to~$M$ is
$\bar M(\bar x)=\frac{a-\bar x}{a-1}$.  Taking into account that
\begin{align*}
\bar Q&=\left[\beta Q+(\epsilon-\beta)a+(\delta-\beta)\right]/(\alpha-\gamma+1),  \\
\bar Q'&=\left[\beta' Q'+(\epsilon'-\beta')a'+(\delta'-\beta')\right]/(\alpha'-\gamma'+1),
\end{align*}
and that $q=\alpha\beta Q$ and $q'=\alpha'\beta' Q',$ one finds
$q'=(-q+\gamma\alpha a)/(a-1)$.

The preceding procedure can be carried~out for all six homographies.  It
yields the following bijection between (positively signed) permutations of
the nonzero singular points $x=1,a,\infty$ of the~HE, and M\"{o}bius
transformations of~$\Hl$.
\begin{eqnarray*}
{}[1_+][a_+][\infty_+] & \ \sim\  &\Hl(a,q;\,\alpha,\beta;\,\gamma,\delta;\,x),\vphantom{\tfrac{(1-a) x}{x-a}}\\
{}[1_+\infty_+][a_+] & \ \sim\  &{(1-x)}^{-\alpha}\Hl(\tfrac{a}{a-1},\tfrac{-q+\gamma \alpha a}{a-1};\,\alpha,\alpha-\delta+1;\,\gamma,\alpha-\beta+1;\,\tfrac{x}{x-1}),\vphantom{\tfrac{(1-a) x}{x-a}}\\
{}[1_+][a_+\infty_+] & \ \sim\  &{(1-\tfrac{x}{a})}^{-\alpha}\Hl({1-a},{-q+\gamma \alpha};\,\alpha,-\beta+\gamma+\delta;\,\gamma,\delta;\,\tfrac{(1-a) x}{x-a}),\vphantom{\tfrac{(1-a) x}{x-a}}\\
{}[1_+a_+][\infty_+] & \ \sim\  &\Hl(\tfrac{1}{a},\tfrac{q}{a};\,\alpha,\beta;\,\gamma,\alpha+\beta-\gamma-\delta+1;\,\tfrac{x}{a}),\vphantom{\tfrac{(1-a) x}{x-a}}\\
{}[1_+a_+\infty_+] & \ \sim\  &{(1-\tfrac{x}{a})}^{-\alpha}\Hl(\tfrac{1}{1-a},\tfrac{q-\gamma \alpha}{a-1};\,\alpha,-\beta+\gamma+\delta;\,\gamma,\alpha-\beta+1;\,\tfrac{x}{x-a}),\vphantom{\tfrac{(1-a) x}{x-a}}\\
{}[1_+\infty_+a_+] & \ \sim\  &{(1-x)}^{-\alpha}\Hl(\tfrac{a-1}{a},\tfrac{-q+\gamma \alpha a}{a};\vphantom{\tfrac{(1-a) x}{x-a}}\\
                   &          &\hskip1in \alpha,\alpha-\delta+1;\,\gamma,\alpha+\beta-\gamma-\delta+1;\,\tfrac{(a-1) x}{a (x-1)}).
\end{eqnarray*}
Here, the second transformation is the just-derived one.  It and the third
are Pfaff-type transformations of~$\Hl(a,q;\alpha,\beta;\gamma,\delta;x)$.
The fourth is a simple `$\delta\leftrightarrow\epsilon$' one, but the fifth
and sixth are quite novel, having no ${}_2F_1$ analogue.  They cyclically
permute the nonzero singular points, so they are not involutions: they are
of order~$3$.

The full transformation group of~$\Hl$ is isomorphic to~$\mathcal{D}_3\cong
S_4,$ or to~$\mathcal{B}_3\cong\mathbb{Z}_2\times S_4$ if the
$\alpha\leftrightarrow\beta$ interchange is included.  It is generated by
the M\"{o}bius transformations, and by F-homotopic ones of the
type~(\ref{eq:nontrivial}), which are easier to work~out explicitly
(details are left to the reader).  The simplest such is
\begin{displaymath}
  [1_-][a_+][\infty_-]\quad\sim\quad (1-x)^{1-\delta}\,\Hl
\bigl(a,q-(\delta-1)\gamma a;\, \beta-\delta+1,\alpha-\delta+1; \,
\gamma,2-\delta; \, x\bigr),
\end{displaymath}
coming from an F-homotopy at~$x=1,$ which is analogous
to~(\ref{eq:Euler2}), Euler's transformation of~${}_2F_1$.  An explicit
bijection between the order-$24$ group $\mathcal{D}_3$ and the
transformations of $\Hl(a,q;\alpha,\beta;\gamma,\delta;x)$ is given
in~\cite[Table~2]{Maier10}.  The $24$ transformations were obtained by
machine computation.  But as the present treatment makes clear, they can
also be obtained reliably by hand.

The automorphism group of the HE is isomorphic to~$\mathcal{D}_4$ (of
order~$192$), or to~$\mathcal{B}_4$ (of order~$384$), if
$\alpha\leftrightarrow\beta$ is included.  The entire set of $192$ HE
solutions appears in Ref.~\cite{Maier10}, and the bijection
to~$\mathcal{D}_4$ is given explictly.  The $192$ are partitioned into
$8$~equivalence classes, one of which comprises the just-mentioned
$24$~transformed versions of~$\Hl$.  The remaining~$7$ come from the cosets
of~$\mathcal{D}_3$ in~$\mathcal{D}_4$.  For later use, one of the
$168=192-24$ solutions of the latter sort will be given here; namely,
\begin{equation}
\label{eq:asoln}
[a_+0_+1_+][\infty_+]\quad\sim\quad\Hl(\tfrac{a-1}a,\tfrac{-q+\beta\alpha a}{a};\,\alpha,\beta;\,\alpha+\beta-\gamma-\delta+1,\gamma;\,\tfrac{a-x}{a}).
\end{equation}
This is obtained from a M\"obius lifting
$M\colon\mathbb{P}_x^1\to\mathbb{P}_{x'}^1,$ where $M(x)\defeq
\frac{a-x}{a}$.  Since $M^{-1}(0)=a\neq0,$ it is not a transformed version
of~$\Hl$: it is a zero-exponent solution of the HE near the singular point
$x=a$.  It and $\Hl(a,q;\alpha,\beta;\gamma,\delta;x)$ do~not equal each
other on their common domain of definition, if this domain is nonempty.
\end{myexample*}

\section{Quadratic and quartic transformations of~$\Hl$}
\label{sec:quadratics}

The transformations of~$\Hl$ derived in~\S\,\ref{sec:autos} were based on
M\"{o}bius transformations of~$\mathbb{P}^1,$ i.e., on degree-$1$ covering
maps $M\colon\mathbb{P}^1_x\to \mathbb{P}^1_{x'}$.  But if
$R\colon\mathbb{P}^1_x\to \mathbb{P}^1_{x'}$ is any rational map, a~HE
$\mathcal{L}u=0$ on the Riemann sphere~$\mathbb{P}^1_{x'}$ will lift
along~$R$ to an FDE $(R^*\mathcal{L})u=0$ on~$\mathbb{P}_x^1$.  The
characteristic exponents of the operator~$R^*\mathcal{L}$ are determined by
Prop.~\ref{prop:lifting}.  If $R^*\mathcal{L}$ has only four singular
points, the lifted FDE will be a~HE in its own right.  Such a lifting is
said to be a degree-$(\deg R)$ transformation of the~HE\null.  If,
moreover, $R(0)=0,$ then a formula relating a local solution (at~$x=0$) of
the form $\Hl(x),$ to one of the form $\Hl(x')=\Hl(R(x)),$ will result.

Erd\'elyi~\cite[\S15.3]{Erdelyi53} pointed~out the existence of quadratic
and quartic transformations of the~HE, though he did not derive the
corresponding functional equations satisfied by the special function~$\Hl$.
He~showed that if any two of the singular points $x'=0,1,a',\infty$ of the
HE on~$\mathbb{P}_{x'}^1$ have exponent difference~$\frac12,$ then there is
a quadratic lifting to a~HE on~$\mathbb{P}_x^1$.  (The singular points
$x=0,1,a,\infty$ of the lifted~HE may differ, since $a'\neq a$ in general.)
For instance, one has
\begin{multline}
\label{eq:erdelyiP}
(1-x)^{\alpha}\left\{\!
\begin{array}{cccc|c}
0 & 1 & a & \infty & x \\
\hline
0 & 0 & 0 & 2\alpha & \\
1-\gamma & \gamma-2\alpha & 1-\gamma & \gamma & \\
\end{array}
\!\right\}\\
=
\left\{\!
\begin{array}{cccc|c}
0 & 1 & a' & \infty & R(x) \\
\hline
0 & 0 & 0 & \alpha & \\
1-\gamma & \frac12 & \frac12 & \gamma-\alpha & \\
\end{array}
\!\right\},
\end{multline}
in which the map $x'=R(x)$ is of the form
\begin{equation}
\label{eq:Rdef}
  R(x) = A\,\frac{x(a-x)}{1-x},
\end{equation}
provided that $a,a'\in\mathbb{P}^1\setminus\{0,1,\infty\}$ are related by
\begin{equation}
\label{eq:modular}
  a^2(1-a')^2 - 16(1-a)a'=0,
\end{equation}
and the `multiplier'~$A=A(a,a')$ is given by the formula
\begin{equation}
\label{eq:Adef}
  A = A(a,a') = \frac{1+a'}{2(2-a)}.
\end{equation}
The P-symbol identity~(\ref{eq:erdelyiP}) follows from
Prop.~\ref{prop:lifting}, since $R^{-1}(0)=\{0,a\}$ and
$R^{-1}(\infty)=\{1,\infty\}$; and as one can check, $x'=1,a'$ are critical
values of the map $x'=R(x),$ the corresponding critical points being mapped
doubly to~them.  The latter requirement, which causes the singular points
$x'=1,a'$ to `disappear' when lifted to~$\mathbb{P}_x^1,$ is the source of
the quadratic constraint~(\ref{eq:modular}) and the
formula~(\ref{eq:Adef}).

The Riemann surface defined by~(\ref{eq:modular}) is of genus zero.
Equivalently, $a,a'$ can be expressed in~terms of a complex parameter~$t$.
By examination, one can choose
\begin{equation}
  a(t)=\frac{t(t+8)}{(t+4)^2},\qquad a'(t)=\frac{t^2}{(t+8)^2};\qquad
  A(t)=\left(\frac{t+4}{t+8}\right)^2.
\end{equation}
The choices $t=0,-4,-8,\infty$ are `unphysical', since each corresponds to
one of~$a,a'$ equaling one of~$0,1,\infty$.  So, the possible quadratic
Heun liftings of the form~(\ref{eq:erdelyiP}) are parametrized by the
quadruply punctured sphere $\mathbb{P}^1_t\setminus\{0,-4,-8,\infty\}$.

The quadratic (and biquadratic quartic) transformations of~$\Hl$ can now be
determined.  To avoid degenerate cases, it will be assumed that the
underlying rational map $R\colon\mathbb{P}^1_x\to\mathbb{P}_{x'}^1$ has the
property that each critical value, on~$\mathbb{P}_{x'}^1,$ is one of the HE
singular points $x'=0,1,a,\infty$.

\begin{mytheorem}
\label{thm:quadratic}
  There is a \emph{unique} quadratic transformation of~$\Hl,$ up~to pre-
  and post-composition with M\"{o}bius and F-homotopic transformations
  of~$\Hl$.  It~is
  \begin{align*}
    &\Hl(a,q;\,2\alpha,\gamma;\,\gamma,2\alpha-\gamma+1;\,x)\\
    &\qquad = (1-x)^{-\alpha}
    \Hl\bigl(a',A(q-\gamma\alpha a);\,\alpha,\gamma-\alpha;\,\gamma,\tfrac12;\,R(x)\bigr),
  \end{align*}
where $a,a'$ are related by\/~{\rm(\ref{eq:modular}),} and the quadratic
map~$R$ and multiplier~$A$ are defined
by\/~{\rm(\ref{eq:Rdef}),(\ref{eq:Adef}).}  This equality holds on a
neighborhood of~$x=0$.
\end{mytheorem}
\begin{proof}
  Any quadratic map~$R$ must have branching schema $1+1=1+1=2=2$.
  (A~schema of this sort, with four `slots', lists the multiplicities with
  which the inverse images of the four singular points $x=0,1,a',\infty$
  on~$\mathbb{P}_{x'}^1$ are mapped to them; orderings of slots and inverse
  images are not significant.)  In~general, the lifted HE will have six
  singular points.  The number will be reduced to four, by
  Prop.~\ref{prop:lifting}, if each of the two singular points
  on~$\mathbb{P}_{x'}^1$ listed last has exponent difference~$\frac12$.

  Let the schema be annotated by underscoring each inverse image that is a
  singular point of the operator~$R^*\mathcal{L}$ on~$\mathbb{P}^1_x$.
  Then there is a unique annotated schema corresponding to a lifting of an
  HE to a~HE, namely
  $\underline1+\underline1=\underline1+\underline1=2=2$.  For this to yield
  a transformation of~$\Hl,$ one must have $R(0)=0$.  So $x'=0$ must be
  assigned to one of the first two slots.  The remaining singular points
  $x'=1,a',\infty$ can be assigned arbitrarily to the remaining three
  slots; and the singular points $x=1,a,\infty$ to the remaining inverse
  images of type~$\underline1$.  Any such arrangment can be reduced by pre-
  and post-composition with M\"obius transformations to the
  scheme~(\ref{eq:erdelyiP}), except that F-homotopic transformations may
  also be applied.

  The $\Hl$ identity of the theorem can be read~off
  from~(\ref{eq:erdelyiP}), except for the value $q'=A(q-\gamma\alpha a)$
  of the right-hand accessory parameter, which follows by direct
  computation: applying the change of variable $x'=R(x)$ to the~HE
  $\mathcal{L}u=0$.
\end{proof}
\begin{myremarkaftertheorem}
  The full set of quadratic transformations of~$\Hl$ is obtained by
  (i)~composing with M\"{o}bius maps, i.e., using $M_2\circ R\circ M_1$
  rather than~$R$ as the covering map (the left side of the identity being
  transformed by~${M_1}^{-1},$ and the right by~$M_2$), and (ii)~acting on
  either side by F-homotopies.  Quadratic transformations are therefore
  bijective to the elements of $\mathcal{D}_3\times\mathcal{D}_3,$ modulo
  an appropriate stabilizing subgroup.  (Here
  `$\alpha\leftrightarrow\beta$' is disregarded.)  A~complete list will
  appear elsewhere.
\end{myremarkaftertheorem}
\begin{myremarkaftertheorem}
  To place Theorem~\ref{thm:quadratic} in context, there is no~analogous
  uniqueness result for the quadratic transformations of the Gauss
  hypergeometric function~${}_2F_1$.  There is a unique annotated schema
  for a quadratic lifting of a GHE on~$\mathbb{P}_{x'}^1$ to a GHE
  on~$\mathbb{P}_x^1$; namely, $\underline1+\underline1=\underline2=2$.
  But there are two ways of assigning $x'=0$: it~can be assigned to the
  first slot or the second, and the resulting quadratic maps $R(x)$ are
  distinct (up~to pre- and post-composition with M\"obius transformations,
  i.e., with Pfaff's involution $x\mapsto x/(x-1)$).  The two choices yield
  \begin{displaymath}
    \begin{array}{ccccc}
      R(x)\quad =& 4x(1-x), & \frac{-4x}{(1-x)^2},  & \frac{4x}{(1+x)^2}, &  \frac{4x(1-x)}{(1-2x)^2} \,;\\
      R(x)\quad =& \frac{x^2}{(2-x)^2}, & \frac{-x^2}{4(1-x)}, &  &  \\
    \end{array}
  \end{displaymath}
  respectively.  Taking pre- and post-composition with F-homotopies (i.e.,
  with Euler's transformation) into account, one sees that the
  transformations coming from the first alternative are bijective to the
  elements of the group~$\mathcal{D}_2\times\mathcal{D}_2,$ and those from
  the second are bijective to the same group, modulo a stabilizing subgroup
  of order~$2$.  A~partial table of quadratic transformations of~${}_2F_1$
  is given in~\cite[\S2.11]{Erdelyi53}.  The existence of two distinct
  types has been stressed by Askey~\cite{Askey94}.
\end{myremarkaftertheorem}

Erd\'elyi also pointed~out that if \emph{three} of the exponent-differences
of the HE on~$\mathbb{P}_{x'}^1$ are equal to~$\frac12,$ such as those
at~$x'=1,a',\infty,$ then it can be lifted to a~HE along three distinct
degree-$2$ covering maps $x'=R(x),$ and each can be followed by a second
degree-$2$ lifting, leading to a \emph{biquadratic quartic transformation}.
By examination, any such composition of two quadratic maps, $x'=S(x),$ has
branching schema
$\underline1+\underline1+\underline1+\underline1=2+2=2+2=2+2$.  This schema
is exemplified by
\begin{equation}
\label{eq:quarticlift}
  S(x)=\frac{4a\,x(1-x)(a-x)}{(a-x^2)^2}
  =1-\frac{(a-2ax+x^2)^2}{(a-x^2)^2} = a-\frac{a(a-2x+x^2)^2}{(a-x^2)^2},
\end{equation}
for which $a=a'$; for this map, $S^{-1}(0)=\{0,1,a,\infty\}$.  The schema
is so symmetrical that when deriving quartic transformations of~$\Hl$
from~it, there is essentially only one way of assigning $x'=0,1,a',\infty$
to the slots, and $x=0,1,a,\infty$ to the inverse images of
type~$\underline1$.  Proceeding as above, and computing the accessory
parameter~$q$ of the lifted~HE from~$q'$ by direct computation, one obtains
the following.

\begin{mytheorem}
\label{thm:biquadratic}
  There is a \emph{unique} biquadratic quartic transformation of~$\Hl,$
  up~to pre- and post-composition with M\"{o}bius and F-homotopic
  transformations of~$\Hl$.  Valid on a neighborhood of~$x=0,$ it is
  \begin{align*}
    &\Hl(a,q;\,2\gamma-1,\gamma;\,\gamma,\gamma;\,x)\\
    &\qquad = (1-x^2/a)^{-\gamma+\frac12}
    \Hl\bigl(a,q/4;\,\tfrac{\gamma}2-\tfrac14,\tfrac{\gamma}2+\tfrac14;\,\gamma,\tfrac12;\,S(x)\bigr),
  \end{align*}
where the biquadratic quartic map~$S$ is defined
by\/~{\rm(\ref{eq:quarticlift}).}
\end{mytheorem}

If $\gamma=1/2,$ then all four of the exponent-differences of the HE are
equal to~$\frac12$.  In this quite degenerate case, the biquadratic
$\Hl$~identity specializes as~follows.

\begin{mycorollary}
The special function $H(a,q;x)\defeq \Hl(a,q;0,\frac12;\frac12,\frac12;x),$
where $a\in\mathbb{C}\setminus\{0,1\}$ and~$q\in\mathbb{C},$ which is
defined on~$\left|x\right|<\min(1,\left|a\right|),$ satisfies the
functional equation $H(a,q;x)=H(a,q/4;S(x))$ on a neighborhood of~$x=0$.
\end{mycorollary}
\begin{myremark*}
  By clever changes of variable, due to L.~Carlitz (see
Valent~\cite{Valent86}), one can show
$H(a,q;x)=\cosh\left[2\sqrt{q/a}\,\,{\rm sn}^{-1}(\sqrt{x}\,)\right],$
where ${\rm sn}^{-1}$~is the inverse of the Jacobi elliptic function~${\rm
sn}(\cdot\,|\,1/a),$ and the branch that vanishes at $x=0$ is meant.  So
the corollary says that ${\rm sn}^{-1}(\sqrt{x}) = {\rm
sn}^{-1}\bigl(\sqrt{S(x)}\bigr)/2,$ i.e., that ${\rm sn}^2(2u\,|\,1/a) =
S\left({\rm sn}^2(u\,|\,1/a)\right)$.  That is, the $\gamma=1/2$ case of
the biquadratic $\Hl$ transformation is really the duplication formula for
the Jacobi function~\cite[\S13.17]{Erdelyi53}, in disguise.  That the
biquadratic quartic function~$S$ plays a dual role is not well~known.
\end{myremark*}

\section{From $\Hl$ to~${}_3F_2$}
\label{sec:Hlto3F2}

The local Heun function $\Hl(x)\defeq \sum_{n=0}^\infty c(n)x^n$ has been less
extensively investigated than the Gauss hypergeometric
function~${}_2F_1(x),$ or even the generalized hypergeometric functions
${}_3F_2(x),$ ${}_4F_3(x),$ etc.  The presence of a fourth singular point
$x=a$ in the HE~(\ref{eq:HE}), or equivalently, the fact that the
recurrence~(\ref{eq:HErec}) satisfied by the coefficient sequence
$c\colon\mathbb{N}\to\mathbb{C}$ is of second order rather than first, is a
novel and somewhat disconcerting feature.  It is natural to ask whether,
for certain parameter choices, $\Hl$~can be reduced to more familiar
transcendental functions.

That $\Hl$ may reduce to~${}_3F_2$ was shown by Letessier et
al.~\cite{Letessier94}.  In this section their result is rederived from
scratch, starting from the `$\Hl$ side' rather than the `${}_3F_2$ side,'
as it were.  A~connection to P-symbols will be made.  As usual, it will be
assumed that the parameter~$\gamma$ is not a nonpositive integer.

First, consider how the singular point $x=a$ of~(\ref{eq:HE}) can be
removed or otherwise tamed.  The simplest case is $\epsilon=0,$ when the
characteristic exponents of~$x=a$ are $0,1-\epsilon=1,$ which are those of
an ordinary point.  (The case $\epsilon=2,$ which is related to this one by
an F-homotopy, is left to the reader.)  Since the difference between the
two exponents is an integer, the local solution corresponding to the
smaller one, zero, will generically be logarithmic.  But a glance
at~(\ref{eq:HE}) reveals that if $q=\alpha\beta a$ when $\epsilon=0,$ then
the HE will reduce to the GHE\null.  Since
$\delta=\alpha+\beta-\gamma-\epsilon+1,$ one has the reduction
$\Hl(a,\alpha\beta
a;\alpha,\beta;\gamma,\alpha+\beta-\gamma+1;x)={}_2F_1(\alpha,\beta;\gamma;x)$.

The next simplest case is $\epsilon=-1,$ when the exponents of~$x=a$ are
$0,2$.  (The case $\epsilon=3,$ related to this one by an F-homotopy, is
also left to the reader.)  As before, the zero-exponent local solution of
the HE at~$x=a$ will generically be logarithmic.  If it is not, then by
definition $x=a$ will be a (simple) {\em apparent\/} singular point: all
local solutions of the HE at~$x=a$ will be analytic.

\begin{mytheorem}
\label{thm:family}
  At fixed $\alpha,\beta,\gamma\in\mathbb{C},$ the family of HEs with
  $\epsilon=-1$ {\rm(}and hence $\delta=\alpha+\beta-\gamma+2${\rm)}, the
  singular point at~$x=a$ being required to be apparent, can be
  parametrized by an auxiliary parameter $e\in\mathbb{C}$ according to
\begin{displaymath}
  a=\frac{e(e-\gamma+1)}{(e-\alpha)(e-\beta)},\qquad 
  q=\alpha\beta\,\frac{(e+1)(e-\gamma+1)}{(e-\alpha)(e-\beta)}.
\end{displaymath}
\end{mytheorem}

\begin{proof}
The singular point $x=a$ will be apparent iff $q$~satisfies a certain
algebraic condition.  It can be worked~out by substituting
$F(x)=\sum_{n=0}^\infty \tilde c(n)(x-a)^n$ into~(\ref{eq:HE}), with
$\tilde c(0)=1,$ and solving for $\tilde c(1),\tilde c(2),$ etc.  In
calculating~$\tilde c(2)$ a division by zero will occur, unless
$q$~satisfies the condition, in which case the equation for~$\tilde c(2)$
will be $0\cdot \tilde c(2)=0$.  The condition is that the right side of
this equation equal zero.

More efficiently, one may exploit the explicit formula~(\ref{eq:asoln}) for
the zero-exponent local solution of the HE at~$x=a$.  When $\epsilon=-1,$
it reduces to
\begin{equation}
\label{eq:relabel}
  \Hl(a',q';\alpha',\beta';\gamma',\delta';x')
  \defeq  
  \Hl(\tfrac{a-1}a,\tfrac{-q+\beta\alpha a}{a};\,\alpha,\beta;\,-1,\gamma;\,\tfrac{a-x}{a}).
\end{equation}
Since $\gamma'=-1,$ this $\Hl$ has exponents $0,2$ at $x'=\frac{a-x}a=0,$
and is generically logarithmic there.  The condition on~$q'$ for
$\Hl(a',q';\alpha',\beta';-1,\delta';x')$ to be non-logarithmic follows
readily from the recurrence~(\ref{eq:HErec}).  One finds that the equation
determining~$c(2)$ is of the form $0\cdot c(2)=0,$ and therefore has a
solution, iff
\begin{equation}
\label{eq:qprime}
  {q'}^2 + \left[(\alpha'+\beta'-\delta'+1) + (\delta'-1)a'\right]q' + \alpha'\beta' a' = 0.
\end{equation}
Substituting the expressions for $a',q';\alpha',\beta';\delta'$
from~(\ref{eq:relabel}) into~(\ref{eq:qprime}) yields
\begin{equation}
\label{eq:aqcurvedef}
q^2+\left[(\gamma-1)-(2\alpha\beta+\alpha+\beta)a \right]q + \alpha\beta a\left[(\alpha\beta+\alpha+\beta+1)a-\gamma\right]=0,
\end{equation}
which is the desired condition on~$q$.

If $\alpha,\beta,\gamma$ are fixed, Eq.~(\ref{eq:aqcurvedef}) defines a
curve in the $a$\textendash$q$ plane.  (In the complex domain, it is a
two-sheeted Riemann surface over $\mathbb{P}^1_a$.)  Being quadratic, it is
of genus zero and can be uniformized by a $\mathbb{P}^1$-valued parameter,
say~$e$.  By examination, the parametrizations $a=a(e),$ $q=q(e)$ supplied
in the theorem will work.
\end{proof}
\begin{myremarkaftertheorem}
The theorem says that at fixed $\alpha,\beta,\gamma,$ the family of HEs for
which $x=a$ has exponents $0,2,$ but is a relatively `tame' apparent
singular point, is parametrized by the Riemann sphere~$\mathbb{P}_e^1$.
Generically this sphere is sextuply punctured, since $e=0,\gamma-1$ yield
$a=0,$ $e=\alpha\beta/(\alpha+\beta-\gamma+1),\infty$ yield $a=1,$ and
$e=\alpha,\beta$ yield $a=\infty$; and necessarily $a\neq0,1,\infty$.
\end{myremarkaftertheorem}

\begin{mydefinition}
The special function $G(x)\defeq G(\alpha,\beta;\gamma;e;x)$ is defined to equal
$\Hl(a,q;\alpha,\beta;\gamma,\delta;x)$ on a neighborhood of $x=0$.  Here
$\delta=\alpha+\beta-\gamma+2$ (so~that $\epsilon=-1$), and $a,q$~are given
in~terms of~$e$ (and $\alpha,\beta,\gamma$) in Theorem~\ref{thm:family}.
\end{mydefinition}

How the parametrized local functions $G(x),{}_3F_2(x)$ are related will now
be explained.  The following facts about~${}_3F_2(x)$ will be needed.  It
can be defined as a solution of a canonical third-order differential
equation, or as a sum $\sum_{n=0}^\infty d(n)x^n,$ where
$d\colon\mathbb{N}\to\mathbb{C}$ satisfies a first-order
recurrence~\cite[Ch.~4]{Erdelyi53}.  The differential equation is
$\bigl[xD\prod_{i=1}^2(xD+b_i-1)-x\prod_{i=1}^3(xD+a_i)\bigr]F=0$; or
equivalently,
\begin{sizemultline}{\small}
\label{eq:ClausenDE}
\biggl\{D^3 + \left[\frac{a_1+a_2+a_3+3}{x-1}-\frac{b_1+b_2+1}{x(x-1)}\right] D^2 \\
{}+\left[\frac{a_1a_2 + a_2a_3 + a_3a_1 + a_1 + a_2 + a_3 + 1}{x(x-1)} - \frac{b_1b_2}{x^2(x-1)}\right] D
+\left[\frac{a_1a_2a_3}{x^2(x-1)}\right]\biggr\}F = 0.
\end{sizemultline}
Here $a_1,a_2,a_3;b_1,b_2$ are complex-valued parameters, and the equation
is invariant under separate permutations of $a_1,a_2,a_3$ and~$b_1,b_2$.
The singular points are $x=0,1,\infty,$ with respective characteristic
exponents
$0,\nobreak{1-b_1},\nobreak{1-b_2};\allowbreak0,\nobreak1,\nobreak
s;\allowbreak a_1,\nobreak a_2,\nobreak a_3,$ where
$s\defeq b_1+b_2-a_1-a_2-a_3$ is the `parametric excess.'  If either of
$b_1,b_2$ is a nonpositive integer, the local solution
of~(\ref{eq:ClausenDE}) corresponding to the zero exponent at~$x=0$ will
generically be logarithmic; but if neither is a nonpositive integer, which
will be assumed henceforth, the solution will be analytic
on~$\left|x\right|<1$.  When normalized to unity at~$x=0,$ it is denoted
${}_3F_2(x)\defeq {}_3F_2(a_1,a_2,a_3;b_1,b_2;x)$.

Equation~(\ref{eq:ClausenDE}) is not the most general third-order (linear,
homogeneous) FDE with the above exponents, but it is the most general one
associated to a first-order recurrence.  The coefficient sequence
$d\colon\mathbb{N}\to\mathbb{C}$ of~${}_3F_2(x)$ satisfies
\begin{equation}
\Bigl\{(n+b_1)(n+b_2)(n+1)E - (n+a_1)(n+a_2)(n+a_3)\Bigr\}\,d=0,
\label{eq:ClausenDErec}  
\end{equation}
initialized by~$d(0)=1$.  Here $E$~is the shift operator.  This
hypergeometric recurrence is of~course the traditional way of
defining~${}_3F_2(x)$.

\begin{mytheorem}
\label{thm:latermain}
  The local function $G(\alpha,\beta;\gamma;e;x)$ can be expressed in~terms
of~${}_3F_2,$ as ${}_3F_2(\alpha,\beta,e+1;\gamma,e;x)$.
\end{mytheorem}
\begin{proof}
It is readily verified that if $\epsilon=-1$ (so that
$\delta=\alpha+\beta-\gamma+2$), and $a,q$~are given the values supplied in
Theorem~\ref{thm:family}, then the second-order difference operator in the
recurrence equation~(\ref{eq:HErec}), the normalized solution of which is
the coefficient sequence $c\colon\mathbb{N}\to\mathbb{C}$ in the series
expansion of~$G,$ has the factorization
\begin{multline}
    \label{eq:Efac}
  \tfrac1{n+e+1}\,\Bigl\{\tfrac{e(e-\gamma+1)}{(e-\alpha)(e-\beta)}E - 1\Bigr\}\\
  {}\times\Bigl\{(n+\gamma)(n+e)(n+1)E - (n+\alpha)(n+\beta)(n+e+1)\Bigr\}.
\end{multline}
The right factor is the operator appearing in~(\ref{eq:ClausenDErec}), with
$a_1,a_2,a_3;b_1,b_2$ set to $\alpha,\beta,e+1;\gamma,e$.  So the
hypergeometric sequence $d\colon\mathbb{N}\to\mathbb{C}$ with these
parameters is a solution of~(\ref{eq:HErec}).  Its initial conditions are
those of~$c,$ i.e., $d(0)=1,\allowbreak d(-1)=0$.  So $c=d,$ and
$G(x)={}_3F_2(x)$.
\end{proof}
\begin{proof}[Alternative Proof]
  It is readily verified that if $a_1,a_2,a_3;b_1,b_2$ are set to
  $\alpha,\beta,e+1;\gamma,e,$ then the third-order differential operator
  in Eq.~(\ref{eq:ClausenDE}), the unique (normalized) analytic local
  solution of which is~${}_3F_2,$ has the factorization
  \begin{multline}
    \label{eq:Dfac}
    \left\{D + \left[\frac{e+1}x+\frac1{x-1}+\frac1{x-a}\right]\right\}\\ 
    {}\times \biggl\{D^2 + \left[ \frac\gamma x + \frac\delta{x-1} +
    \frac\epsilon{x-a} \right]D + \left[\frac{\alpha\beta\, x -
    q}{x(x-1)(x-a)}\right]\biggr\},
  \end{multline}
  in which $\epsilon\defeq -1,$ $\delta\defeq \alpha+\beta-\gamma+2,$ and
  $a,q$ are defined as in Theorem~\ref{thm:family}.  The right factor is
  the operator appearing in the HE~(\ref{eq:HE}).  So, ${}_3F_2(x)=G(x)$.
\end{proof}
\begin{myremark*}
The operator factorization (\ref{eq:Efac}), resp.~(\ref{eq:Dfac}), can be
viewed as taking place in the noncommutative polynomial ring generated by
$n,E,$ resp.\ by~$x,D$.  (The noncommutativity is specified by $[E,n]=E,$
resp.\ $[D,x]=1$.)  Such factorizations, when they exist, can be found by
noncommutative Gr\"obner basis techniques~\cite{Melenk94}.  There are also
specialized algorithms for finding hypergeometric-sequence solutions to
recurrences, i.e., for finding first-order right factors of difference
operators~\cite{Cluzeau2006,Petkovsek92}.  But in~general, it is difficult
to characterize when a difference or differential operator will factor.  If
$\epsilon=-1$ and $(a,q)$ lies on the algebraic
curve~(\ref{eq:aqcurvedef}), then the second-order difference operator
associated with~$\Hl$ will factor, as shown in~(\ref{eq:Efac}).
Conversely, the factoring of the third-order differential operator
associated with~${}_3F_2(a_1,a_2,e+1;b_1,e;x),$ as in~(\ref{eq:Dfac}),
yields a parametrization of the $a$\textendash$q$ curve.
\end{myremark*}

\begin{mycorollary}
  \label{cor:two2F1s}
  $G(\alpha,\beta;\gamma;e;x)$ can be expressed in~terms of the Gauss
  function ${}_2F_1,$ as
  ${}_2F_1(\alpha,\beta;\gamma;x)+(\alpha\beta/\gamma
  e)x\,{}_2F_1(\alpha+1,\beta+1;\gamma+1;x),$ or equivalently, in~terms of
  ${}_2F_1$ and its derivative as
  $e^{-1}[xD+e]\,{}_2F_1(\alpha,\beta;\gamma;x)$.  It can also be written
  as $[(e-\gamma+1)/e]\,{}_2F_1(\alpha,\beta;\gamma;x) +
  [(\gamma-1)/e]\,{}_2F_1(\alpha,\beta;\gamma-1;x)$.
\end{mycorollary}
\begin{proof}
  For any generalized hypergeometric function, the case when an upper
  parameter exceeds a lower one by an integer is special: it can be
  reduced.  For instance~\cite[\S5.2]{Luke75},
  ${}_3F_2(a_1,a_2,e+1;b_1,e;x)$ can be written in~terms of the Gauss
  function, as
  ${}_2F_1(a_1,a_2;b_1;x)+(a_1a_2/b_1e)x\,{}_2F_1(a_1+1,a_2+1;b_1+1;x)$,
  i.e., as $e^{-1}[xD+e]\,{}_2F_1(a_1,a_2;b_1;x)$.  This can be verified by
  comparing series expansions.  It is also the case that the general
  function $F:={}_3F_2(a_1,a_2,a_3;b_1,b_2;x)$ satisfies the three-term
  contiguity relation
  \begin{equation}
    F(a_3+)=[(a_3-b_1+1)/a_3]\,F + [(b_1-1)/a_3]\,F(b_1-),
  \end{equation}
  where $a_3+$ and $b_1-$ indicate incrementing and decrementing of
  parameters~\cite{Rainville45}.  Setting
  $(a_1,a_2,a_3;b_1,b_2)=(\alpha,\beta,e;\gamma,e)$ yields the second
  representation.
\end{proof}

By Theorem~\ref{thm:latermain}, the function $G(x)$ has both an $\Hl$ and a
${}_3F_2$ representation.  Their relationship is revealed by their
P-symbols, which are respectively
\begin{equation}
\left\{\!
\begin{array}{cccc|c}
0 & 1 & a & \infty & x \\
\hline
0 & 0 & 0 & \alpha & \\
1-\gamma & 1-\delta & 2 & \beta & \\
\end{array}
\!\right\},
\qquad
\left\{\!
\begin{array}{cccc|c}
0 & 1 & a & \infty & x \\
\hline
0 & 0 & 0 & \alpha &  \\
1-\gamma & 1-\delta & 2 & \beta & \\
\hline
1-e & 1 & 1 & e+1 & \\
\end{array}
\!\right\},
\end{equation}
with $\delta\defeq\alpha+\beta-\gamma+2$.  In the third-order P-symbol on
the right, the ordinary point $x=a$ is listed explicitly, with exponents
$0,1,2$.  (Any order-$k$ FDE has exponents $\{0,\dots,k-1\}$ at any finite
ordinary point.)  The final exponent row comes from the left-multiplication
by a first-order differential operator that was performed
in~(\ref{eq:Dfac}), but the two P-symbols are otherwise the same.
Appending the `$1$' to the $x=a$ column converts $x=a$ to an ordinary
point, leaving only three singular points.

This reduction of $\Hl$ to ${}_3F_2$ is therefore an example of {\em
desingularization\/}: removing an apparent singular point from a
differential operator by left-multiplying by an auxiliary, appropriately
chosen operator.  Such desingularizations may be constructed
algorithmically, as Abramov et~al.~\cite{Abramov2006} show.  They also show
how to desingularize difference operators, which in a certain sense have
singular points too.  Any order-$d$ operator $L\defeq \sum_{j=0}^d
p_j(n)E^j,$ where the~$p_j$ are polynomials with $p_d\neq0,$ may be viewed
as acting on functions defined on~$\mathbb{C}_n,$ the entire complex
$n$-plane.  The difference equation $LF=0$ will have $d$~independent
solutions that are meromorphic on~$\mathbb{C}_n$ and analytic on a common
left half-plane.  The (leading) singular points of~$L$ on~$\mathbb{C}_n$
are the roots of the polynomial $p_d(n-d),$ by definition; they are
possible pole locations.  If $n=n_0$ is a root, but none of the
$d$~solutions has a pole there, it is said to be apparent.
A~desingularized version of~$L$ can be obtained by, e.g., left-multiplying
$L$ by an appropriate difference operator.

The right factor in~(\ref{eq:Efac}), the first-order difference operator
associated to the function ${}_3F_2(\alpha,\beta,e+1;\gamma,e;x),$ has
singular points at $n=0,1-\gamma,1-e$.  It can be shown that the one at
$n=1-e$ is merely apparent.  The left-multiplication in~(\ref{eq:Efac})
removes this apparent singular point.  So, the reduction of this~${}_3F_2$
to~$\Hl$ is also a desingularization.  The two desingularizations are dual
to each other.

\section{${}_3F_2$ transformations}
\label{sec:3F2}

The generalized hypergeometric function
${}_3F_2:={}_3F_2(a_1,a_2,a_3;b_1,b_2;x)$ satisfies many identities, but
its theory is less developed than that of~${}_2F_1$.  It was shown
in~\S\,\ref{sec:Hlto3F2} that the {\em ad~hoc\/} special function
\begin{align*}
G(\alpha,\beta;\,\gamma;\,e;\,x)&={}_3F_2(\alpha,\beta,e+1;\,\gamma,e;\,x)\\
&= e^{-1}[x\,D+e]\,{}_2F_1(\alpha,\beta;\,\gamma;\,x)
\end{align*}
has a Heun representation.  As a result, certain transformation laws
for~$G$ (and hence for a restricted~${}_3F_2$) can be deduced from those
for~$\Hl$.

Pfaff and Euler-like transformations of this ${}_3F_2$ are given in
Theorem~\ref{thm:3F2main} below, which is the main result of this final
section.  Each transformation has four free parameters and a nonlinear
parametric constraint.  Part~(\ref{enum:i}) of the theorem subsumes several
interesting identities discovered by Bailey and Slater, which have fewer
free parameters.  They appear as corollaries of the theorem.

Some terminology from general hypergeometric summation theory will be used.
The function
${}_{r+1}F_r(x):={}_{r+1}F_r(a_1,\dots,a_{r+1};b_1,\dots,b_r;x)$ can be
defined as the solution of a canonical order-$(r+1)$ FDE, or as a sum
$\sum_{n=0}^\infty d(n)x^n$ that converges on~$\left|x\right|<1$.  The FDE
is $\bigl[xD\prod_{i=1}^r(xD+b_i-1)-x\prod_{i=1}^{r+1}(xD+a_i)\bigr]F=0,$
and $d\colon\mathbb{N}\to\mathbb{C}$ satisfies $\bigl[(n+1)\prod_{i=1}^r
(n+b_i)E -\prod_{i=1}^{r+1}(n+a_i)\bigr]d=0$.

In~general, if $\{a_i\}_{i=1}^{r+1}$ and~$\{b_i\}_{i=1}^r$ can be
separately permuted so that $a_1+1=a_2+b_1=\dots=a_{r+1}+b_r,$ then the
function ${}_{r+1}F_r$ is said to be {\em well poised\/}.  If each of these
$r+1$ equalities holds except the last, it is {\em nearly well poised\/},
or simply {\em nearly poised\/}.  If ${}_{r+1}F_r$ is well poised, resp.\
nearly poised, and also one of the $\{b_i\}_{i=1}^r,$ resp.\
$\{b_i\}_{i=1}^{r-1},$ equals~$a_1/2,$ it is {\em very well poised\/},
resp.\ {\em nearly very well poised\/}.

\begin{mytheorem}
\ 
  \begin{enumerate}
    \item\label{enum:i} The function ${}_3F_2$ satisfies the Pfaff-like
    identity
  \begin{align*}
    &{}_3F_2(a_1,a_2,e+1;\,b_1,e;\,x)\\
    &{}= (1-x)^{-a_1}\,
    {}_3F_2(a_1,b_1-a_2-1,e'+1;\,b_1,e';\,\tfrac{x}{x-1}),
  \end{align*}
  in which the $e\leftrightarrow e'$ correspondence is
  \begin{displaymath}
  e'=\frac{(b_1-a_2-1)e}{e-a_2}.
  \end{displaymath}
  \item\label{enum:ii} 
  The function ${}_3F_2$ satisfies the Euler-like identity
  \begin{align*}
    &\qquad\qquad{}_3F_2(a_1,a_2,e+1;\,b_1,e;\,x)\\
    &\qquad\qquad{}= (1-x)^{b_1-a_1-a_2-1}\,
    {}_3F_2(b_1-a_1-1,b_1-a_2-1,e''+1;\,b_1,e'';\,x),
  \end{align*}
  in which the $e\leftrightarrow e''$ correspondence is
  \begin{displaymath}
  e'' = \frac{ (b_1-a_1-1)(b_1-a_2-1)e}{(b_1-a_1-a_2-1)e+a_1a_2}.
  \end{displaymath}
  \end{enumerate}
  \label{thm:3F2main}
  Here $e,e',$ resp.\ $e,e'',$ are taken to be finite.  Also, the M\"obius
  transformation $e\mapsto e',$ resp.\ $e\mapsto e'',$ which is of the
  lower triangular form $e\mapsto \frac{Ae\hphantom{+B}}{Ce+D},$ is assumed
  to be nonsingular, which will be the case if all upper parameters are
  nonzero, i.e., if neither~${}_3F_2$ is identically equal to unity.

\end{mytheorem}
\begin{proof}
  The transformations of~$\Hl(x):=\Hl(a,q;\alpha,\beta;\gamma,\delta;x)$
  analogous to Pfaff's and Euler's transformations of~${}_2F_1$ are,
  from~\S\,\ref{sec:autos},
  \begin{eqnarray*}
    {}[1_+\infty_+][a_+] & \!\!\sim\!\! &{(1-x)}^{-\alpha}\Hl(\tfrac{a}{a-1},\tfrac{-q+\gamma \alpha a}{a-1};\,\alpha,\alpha-\delta+1;\,\gamma,\alpha-\beta+1;\,\tfrac{x}{x-1}),\\
    {}[1_-][a_+][\infty_-] & \!\!\sim\!\! &(1-x)^{1-\delta}\,\Hl
\bigl(a,q-(\delta-1)\gamma a;\, \beta-\delta+1,\alpha-\delta+1; \,
\gamma,2-\delta; \, x\bigr).
  \end{eqnarray*}
  Theorem~\ref{thm:family} gives the parametrization
  $(a,q)=\left(a(e),q(e)\right)$ of the case when the singular point $x=a$
  is a (simple) apparent one.  In this case, $\Hl(x)$ equals
  $G(\alpha,\beta;\gamma;e;x)={}_3F_2(\alpha,\beta,e+1;\gamma,e;x)$.  The
  transformed versions of $\Hl(x)$ have similar parametrizations
  $(a',q')=\left(a'(e'),q'(e')\right)$ and
  $(a'',q'')=\left(a''(e''),q''(e'')\right),$ and may likewise be written
  in~terms of~${}_3F_2$.  Here $a'=\frac{a}{a-1},$
  $q'=\frac{-q+\gamma\alpha a}{a-1},$ and $a''=a,$ $q''=q-(\delta-1)\gamma
  a$.  The transformations therefore yield isomorphisms between genus-zero
  algebraic curves, $(a,q)\leftrightarrow(a',q')$ and
  $(a,q)\leftrightarrow(a'',q''),$ which on the parametric level are
  performed by M\"obius correspondences $e\leftrightarrow e'$
  and~$e\leftrightarrow e''$.  Computing these explicitly yields the
  formulas in the theorem.
\end{proof}
\begin{myremark*}
  The transformations of the theorem can also be verified by expressing
  each ${}_3F_2$ in~terms of two ${}_2F_1$'s, as in
  Corollary~\ref{cor:two2F1s}, and exploiting results on~${}_2F_1$.
  
  The Pfaff and Euler-like transformations generate an order-$4$ group of
  transformations of the family of ${}_3F_2$'s of the form
  ${}_3F_2(a_1,a_2,e+1;b_1,e;x)$.  The group is isomorphic to
  $\mathcal{D}_2\cong\mathbb{Z}_2\times\mathbb{Z}_2$.  (The third
  nontrivial element of this group is a twisted Pfaff transformation,
  obtained from the Pfaff-like one by interchanging $a_1,a_2$.)  If
  $a_1\leftrightarrow a_2$ is included explicitly, the resulting group will
  be isomorphic to~$\mathcal{B}_2,$ i.e., to the dihedral group of
  order~$8$.  This group structure, familiar from~${}_2F_1,$ is inherited
  from the automorphism group of the~HE\null.
\end{myremark*}

\begin{mycorollary}
  \label{cor:corserpentis}
  The function ${}_3F_2$ satisfies the Pfaff-like identity  
  \begin{align*}
    &\qquad{}_3F_2(a_1,a_2,-a_2+1;\,b_1,-a_2;\,x)\\
    &\qquad\qquad{}= (1-x)^{-a_1}\,
    {}_3F_2(a_1,b_1-a_2-1,\tfrac{b_1-a_2+1}2;\,b_1,\tfrac{b_1-a_2-1}2;\,\tfrac{x}{x-1}).
  \end{align*}
\end{mycorollary}
\begin{proof}
  Set $e=-a_2$ in Theorem~\ref{thm:3F2main}(\ref{enum:i}).
\end{proof}
\begin{myremark*}
  This special case, with three free parameters and no nonlinear
  constraint, was discovered by Bailey and
  Slater~\cite[(2.4.2.8)]{Slater66} and rediscovered by
  Baweja~\cite{Baweja81}.  Their derivations were altogether different.
  The right-hand ${}_3F_2$ is nearly very well poised, as one sees by
  permuting parameters.
\end{myremark*}

\begin{mycorollary}
\label{cor:cor0}
  There is a three-parameter reduction of ${}_3F_2$ to a
  single~~${}_2F_1$; namely,
  \begin{displaymath}
    {}_3F_2(a_1,a_2,e+1;\,b_1,e;\,x) = (1-x)\,{}_2F_1(a_1+1,a_2+1;\,b_1;\,x),
  \end{displaymath}
  in which $e:=e(a_1,a_2;b_1)=a_1a_2/(a_1+a_2-b_1+1)$.
\end{mycorollary}
\begin{proof}
  If $e=a_1a_2/(a_1+a_2-b_1+1)$ then $e'=a_1,$ reducing the right-hand side
  in Theorem~\ref{thm:3F2main}(\ref{enum:i}) to
  $(1-x)^{-a_1}\,{}_2F_1(a_1+1,b_1-a_2-1;b_1;\frac{x}{x-1})$.  To obtain
  the given right-hand side, use Eq.~(\ref{eq:pfaff}), Pfaff's
  transformation of~${}_2F_1$.
\end{proof}

\begin{mycorollarycorollary}
\label{cor:cor}
  Any very well poised ${}_3F_2$ can be reduced to a well poised~${}_2F_1,$
  according to
  \begin{displaymath}
    {}_3F_2(\alpha,\beta,\alpha/2+1;\,\alpha-\beta+1,\alpha/2;\,x) = (1-x)\,{}_2F_1(\alpha+1,\beta+1;\,\alpha-\beta+1;\,x).
  \end{displaymath}
\end{mycorollarycorollary}
\begin{proof}
  Set $(a_1,a_2;b_1)=(\alpha,\beta;\alpha-\beta+1)$ in
  Corollary~\ref{cor:cor0}.
\end{proof}
\begin{myremark*}
This reduction of a very well poised ${}_3F_2$ was also found by Bailey.
It is sometimes reproduced in a different but equivalent
form~\cite[(2.4.2.10)]{Slater66}, in which Eq.~(\ref{eq:euler}), Euler's
transformation of~${}_2F_1,$ has been applied to the right side.
\end{myremark*}

\begin{mycorollarycorollary}
\label{cor:cor1}
For any $s,t\in\mathbb{C},$ the parametrized function
\begin{displaymath}
  (1-x)^{2\alpha-1}\,{}_3F_2\Bigl(2\alpha-1,(1-s)\alpha-(1+s)\beta-(\tfrac12-t),e+1;\,
  (1-s)\alpha + (1-s)\beta - (\tfrac12-t), e;\,x\Bigr),
\end{displaymath}
in which 
\begin{displaymath}
  e:=(\alpha-\tfrac12)\left[\frac{(1-s)\alpha-(1+s)\beta-(\frac12-t)}{\alpha-\beta-\frac12}\right],
\end{displaymath}
is stable under the interchange $\alpha\leftrightarrow\beta$.
\end{mycorollarycorollary}
\begin{proof}
  By Corollary~\ref{cor:cor0}, this function reduces to
  \begin{displaymath}
  (1-x)^{2\alpha}\,{}_2F_1\Bigl(2\alpha,(1-s)\alpha-(1+s)\beta+(\tfrac12+t);\,
  (1-s)\alpha + (1-s)\beta + (\tfrac12+t);\,x\Bigr).
  \end{displaymath}
  Applying Eq.~(\ref{eq:euler}), Euler's transformation of~${}_2F_1,$ has
  the effect of interchanging $\alpha$ and~$\beta$ in this expression.
\end{proof}
\begin{myremark*}
  For each $s,t\in\mathbb{C},$ the invariance under
  $\alpha\leftrightarrow\beta$ is really a special case of
  Theorem~\ref{thm:3F2main}(\ref{enum:ii}), the general Euler-like
  transformation of ${}_3F_2(a_1,a_2,e+1;b_1,e;x)$.
\end{myremark*}

\begin{mycorollarycorollary}
There is an involutory transformation of the two-parameter family of very
well poised ${}_3F_2$'s, namely the $\alpha\leftrightarrow\beta$
interchange
\begin{multline*}
  (1-x)^{2\alpha-1}\,{}_3F_2(2\alpha-1,\alpha-\beta-\tfrac12,\alpha+\tfrac12;\,\alpha+\beta+\tfrac12,\alpha-\tfrac12;\,x)\\
  =
  (1-x)^{2\beta-1}\,{}_3F_2(2\beta-1,\beta-\alpha-\tfrac12,\beta+\tfrac12;\,\beta+\alpha+\tfrac12,\beta-\tfrac12;\,x).
\end{multline*}
\end{mycorollarycorollary}
\begin{proof}
Set $s=t=0$ in Corollary~\ref{cor:cor1}.
\end{proof}
\begin{myremark*}
  This transformation law was also discovered by Bailey
  (see~\cite[\S4.5]{Erdelyi53}).  It can be derived directly from
  Corollary~\ref{cor:cor}, which reduces both sides to the same well
  poised~${}_2F_1$.  The generalization given in Corollary~\ref{cor:cor1}
  is new, as is its interpretation as a transformation of Euler type,
  a~special case of Theorem~\ref{thm:3F2main}(\ref{enum:ii}).

  It should be noted that the two sides of this transformation formula are
  solutions of different third-order differential equations.  The solution
  spaces of these equations have nonempty intersection, but are distinct.
  This phenomenon is clarified by the respective P-symbols, which
  by~examination are
  \begin{sizedisplaymath}{\small}
    \left\{\!
    \begin{array}{ccc|c}
      0 & 1 & \infty & x \\
      \hline
      0 & 2\alpha & 0 &  \\
      \tfrac12-\alpha-\beta & 2\beta & \tfrac12-\alpha-\beta & \\
      \hline
      \tfrac32-\alpha & 2\alpha-1 & \tfrac32-\alpha & \\
    \end{array}
    \!\right\},
    \quad
    \left\{\!
    \begin{array}{ccc|c}
      0 & 1 & \infty & x \\
      \hline
      0 & 2\alpha & 0 &  \\
      \tfrac12-\alpha-\beta & 2\beta & \tfrac12-\alpha-\beta & \\
      \hline
      \tfrac32-\beta & 2\beta-1 & \tfrac32-\beta & \\
    \end{array}
    \!\right\}.
  \end{sizedisplaymath}
  The first two rows of exponents are the same in both, but the third row
  differs.  This partial agreement is what one expects to see in two
  third-order operators that are distinct left multiples of a common
  second-order operator.
\end{myremark*}

\end{document}